\documentclass[a4paper,10pt]{article}

\usepackage[T1]{fontenc}
\usepackage[latin1]{inputenc}
\usepackage{palatino, url, multicol}
\usepackage[sc]{mathpazo}
\usepackage{amsthm}
\usepackage{graphicx}
\usepackage{epsfig} 
\usepackage{amsmath}
\usepackage{latexsym}
\usepackage{amssymb}
\usepackage{amscd}
\usepackage{ifthen}
\usepackage{cite}
\usepackage{epstopdf}
\newtheorem{theorem}{Theorem}[section]

\newtheorem{proposition}{Proposition}[section]
\newtheorem{lemma}{Lemma}
\theoremstyle{definition}

\theoremstyle{remark}

\title{On gliding Lagrange top equations and their asymptotic behaviour}
\begin{document}

\author{
	Nils Rutstam\\
  Stefan Rauch-Wojciechowski\\
	Linköping University, Linköping, Sweden
  }

\maketitle

\bigskip

{\abstract The dynamical equations for a gliding Lagrange top are not integrable. They have 5 dynamical variables and admit one integral of motion. We show that all solutions go to one of the two vertical spinning solutions and determine conditions of their stability. This means that solutions starting close to either of the spinning solutions go asymptotically to this solution.}

\medskip 


\section{Equations of motion for the gliding Lagrange top}

We study motion of a spinning and gliding Lagrange top (gLT) of mass $m$ under action of the gravitational force $-mg\hat{z}$ and subjected to a constraint allowing the bottom tip $A$ to glide in a horizontal plane of support. The top may spin above the plane, under the plane and may cross the plane during its motion. Equations for the gliding Lagrange top have been studied in \cite{Nisse2} as a limiting case of equations for the Tippe Top.

For describing motion of the top we use three right-handed reference frames as in Fig 1 \ref{gLT_diagram}. Here $\mathbf{K}_0$ is an inertial frame, $\mathbf{K}=(\mathbf{\hat{1}},\mathbf{\hat{2}},\mathbf{\hat{3}})$ is a (partialy) body fixed frame with origin placed at the centre of mass ($CM$) and having the axis $\mathbf{\hat{3}}$ aligned along the symmetry axis of the top. The frame $\mathbf{\tilde{K}}=(\hat{x},\hat{y},\hat{z})$ has origin at the contact point $A$ and with $\hat{x}$, $\hat{y}$ aligned with the plane of constraint. The $\hat{x}$-axis stays in the vertical plane of ($\mathbf{\hat{1}}$,$\mathbf{\hat{3}}$) and the axis $\mathbf{\hat{2}}$ is parallel to $\hat{y}$. This means that $\mathbf{\tilde{K}}$ is rotating about the vertical axis $\hat{z}$ with angular velocity $\dot{\varphi}$ where $(\theta,\varphi,\psi)$ denote the standard Euler angles describing rotation of the top w.r.t. the inertial reference frame $\mathbf{K}_0$. We let $\boldsymbol{\omega}$ be the angular velocity vector in the $\mathbf{K}$-frame and $\mathbf{L}=\mathbb{I}\boldsymbol{\omega}=I_1\omega_1\mathbf{\hat{1}}+I_1\omega_2\mathbf{\hat{2}}+I_3\omega_3\mathbf{\hat{3}}$ is the angular momentum (due to axial symmetry, we have $I_1=I_2$ for the moments of inertia).
\begin{figure}[ht]
\begin{center}
\includegraphics[scale=0.90]{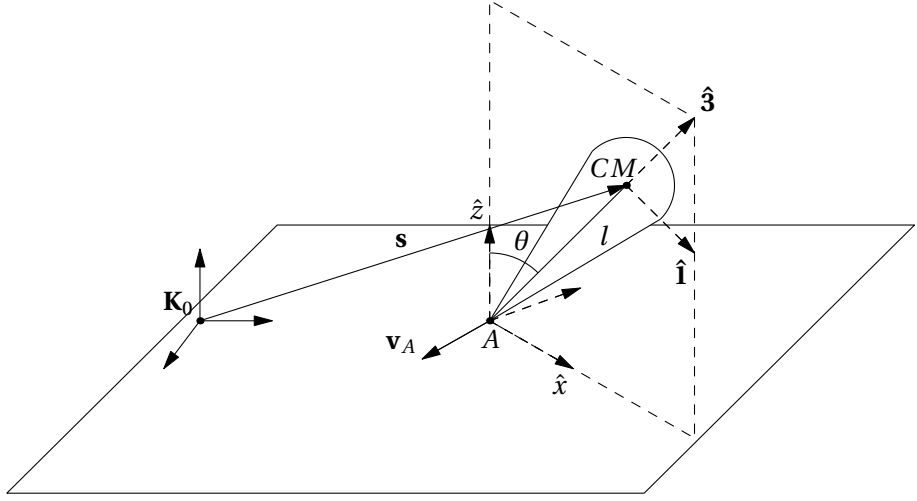}
\caption{Diagram of the gliding Lagrange Top.\label{gLT_diagram}}
\end{center}
\end{figure}
The equations of motion for gLT consist of an equation for the motion of the centre of mass $CM$, an equation for rotation about $CM$ and, due to axial symmetry, of one kinematic equation for motion of the symmetry axis $\mathbf{\hat{3}}$. They are
\begin{equation}
\label{gLT_EoM}m\mathbf{\ddot{s}}=\mathbf{F}-mg\hat{z},\quad \mathbf{\dot{L}}=\mathbf{a\times F},\quad \mathbf{\dot{\mathbf{\hat{3}}}}=\boldsymbol{\omega}\times\mathbf{\hat{3}}=\frac{1}{I_1}(\mathbf{L}\times\mathbf{\hat{3}}),
\end{equation}  
where the dot denotes time-derivative, $\mathbf{s}$ is the position of $CM$ w.r.t. the inertial frame $\mathbf{K}_0$, $\mathbf{a}=-l\mathbf{\hat{3}}$ points from $CM$ to the point of support $A$, $-mg\hat{z}$ is the gravitational force acting at $CM$ and $\mathbf{F}=\mathbf{F}_{\text{f}}+\mathbf{F}_{\text{R}}=-\mu g_n(t)\mathbf{v}_A+g_n(t)\hat{z}$ is the force acting at $A$.
This last force consist of a reaction force $\mathbf{F}_{\text{R}}=g_n(t)\hat{z}$ and a frictional force $\mathbf{F}_{\text{f}}=-\mu g_n(t)\mathbf{v}_A$ acting against the direction of the gliding velocity $\mathbf{v}_A=\mathbf{\dot{s}}+\boldsymbol{\omega}\times\mathbf{a}$. The friction coefficient $\mu(\mathbf{s},\mathbf{\dot{s}},\mathbf{L},\mathbf{\hat{3}},t)>0$ may depend on all dynamical variables and time.

For the classical Lagrange top (LT), for which $A$ is fixed, $\mathbf{v}_A=0$, $\mathbf{\dot{s}}=-\boldsymbol{\omega}\times\mathbf{a}$ and equations of motion are usually given for the angular momentum w.r.t. the point of support $A$. So then $\mathbf{L}_A=\mathbf{L}+m\boldsymbol{\omega}\times(\boldsymbol{\omega}\times\mathbf{a})$ as follows from the Steiner formula. The derivative is $\frac{d}{dt}\mathbf{L}_{A}=\mathbf{a}\times(mg\hat{z})+m(\mathbf{a}\times\mathbf{\dot{v}}_A)$ since $\mathbf{v}_A=\mathbf{\dot{s}}+\boldsymbol{\omega}\times\mathbf{a}$.

The constraint keeping the tip $A$ in the plane means that $(\mathbf{s}(t)+\mathbf{a}(t))\cdot\hat{z}=0$ and that the gliding velocity stays in the plane of support, i.e. $0=(\mathbf{\dot{s}}+\boldsymbol{\omega}\times\mathbf{a})\cdot\hat{z}=\mathbf{v}_{A}\cdot\hat{z}$. The second derivative of the constraint $0=\hat{z}\cdot\frac{d^2}{dt^2}(\mathbf{s+a})=\hat{z}\cdot\left[\mathbf{\ddot{s}}-\frac{1}{I_1}\frac{d}{dt}(\mathbf{L}\times\mathbf{\hat{3}})\right]$ determines the value of the vertical reaction force 
\begin{equation}
\label{gn_gLT}g_n(t)=\frac{mgI_{1}^2+ml(\mathbf{L}\cdot\mathbf{\hat{3}}\mathbf{L}\cdot\hat{z}-\mathbf{\hat{3}}\cdot\hat{z}\mathbf{L}^2)}{I_{1}^2+ml^2I_1(1-(\mathbf{\hat{3}}\cdot\hat{z})^2)+ml^2I_1 \mu\mathbf{\hat{3}}\cdot\hat{z}\mathbf{v}_{A}\cdot\mathbf{\hat{3}}}.
\end{equation}
If we denote $\mathbf{s}=(\mathbf{r},s_{\hat{z}})$ then $s_{\hat{z}}=l\mathbf{\hat{3}}\cdot\hat{z}$, $\dot{s}_{\hat{z}}=l\mathbf{\dot{\mathbf{\hat{3}}}}\cdot\hat{z}=l(\boldsymbol{\omega}\times\mathbf{\hat{3}})\cdot\hat{z}$ are determined and the equation $m\mathbf{\ddot{s}}=\mathbf{F}-mg\hat{z}$ is reduced to $m\mathbf{\ddot{r}}=-\mu g_n(t)\mathbf{v}_A$. 

The system \eqref{gLT_EoM} inherits from LT an integral of motion $\mathbf{L}\cdot\mathbf{\hat{3}}=L_3$ as
\begin{equation*}
\frac{d}{dt}\left(\mathbf{L}\cdot\mathbf{\hat{3}}\right)=\mathbf{\dot{L}}\cdot\mathbf{\hat{3}}+\mathbf{L}\cdot\mathbf{\dot{\mathbf{\hat{3}}}}=(\mathbf{a}\times\mathbf{F})\cdot\mathbf{\hat{3}}+\frac{1}{I_1}\mathbf{L}\cdot(\mathbf{L}\times\mathbf{\hat{3}})=0.
\end{equation*}
The total energy $E=\frac{1}{2}m\mathbf{\dot{s}}^2+\frac{1}{2}\boldsymbol{\omega}\cdot\mathbf{L}+mg\mathbf{s}\cdot\hat{z}$ is not conserved but it is a monotonously decreasing function of time since 
\begin{align*}
\dot{E}=&m\mathbf{\dot{s}}\cdot\mathbf{\ddot{s}}+\boldsymbol{\omega}\cdot\mathbf{\dot{L}}+mg\mathbf{\dot{s}}\cdot\hat{z}=(\mathbf{v}_A-\boldsymbol{\omega}\times\mathbf{a})\cdot\mathbf{F}-mg\mathbf{\dot{s}}\cdot\hat{z}+\boldsymbol{\omega}\cdot(\mathbf{a}\times\mathbf{F})+mg\mathbf{\dot{s}}\cdot\hat{z}\nonumber\\
=&\mathbf{v}_{A}\cdot\mathbf{F}-\mathbf{F}\cdot(\boldsymbol{\omega}\times\mathbf{a})+\boldsymbol{\omega}\cdot(\mathbf{a}\times\mathbf{F})=\mathbf{F}\cdot\mathbf{v}_{A}=-\mu g_n|\mathbf{v}_A|^2
\end{align*}
when we consider mechanical solutions with a positive reaction force.

The projection of $\mathbf{L}_A$ and $\mathbf{L}$ onto the $\hat{z}$-axis is no longer an integral of motion as
\begin{align*}
\frac{d}{dt}(\mathbf{L}_A\cdot\hat{z})&=m(\mathbf{a}\times\mathbf{\dot{v}}_A)\cdot\hat{z},\\
\frac{d}{dt}(\mathbf{L}\cdot\hat{z})&=m(\mathbf{a}\times\mathbf{\dot{v}}_A)\cdot\hat{z}-ml^2\left(\mathbf{\hat{3}}\times\frac{d}{dt}(\boldsymbol{\omega}\times\mathbf{\hat{3}})\right)\cdot\hat{z}.
\end{align*}
The system \eqref{gLT_EoM} thus admits only one integral of motion, it dissipates energy and it is not integrable.

\section{Equations in Euler angles} 

As has been stated, the orientation of the gLT with respect to $\mathbf{K}$ is described by the three angles $(\theta,\varphi,\psi)$, which are functions of time. The angle $\theta$ is the inclination of the symmetry axis $\mathbf{\hat{3}}$ w.r.t. $\hat{z}$, $\varphi$ is the rotation angle around the $\hat{z}$-axis and $\psi$ is the rotation around the $\mathbf{\hat{3}}$-axis.

The angular velocity of the reference frame $(\mathbf{\hat{1}},\mathbf{\hat{2}},\mathbf{\hat{3}})$ with respect to $\mathbf{K}$ is
\begin{equation*}
\boldsymbol{\omega}_{\textrm{ref}}=-\dot{\varphi}\sin\theta\mathbf{\hat{1}}+\dot{\theta}\mathbf{\hat{2}}+\dot{\varphi}\cos\theta\mathbf{\hat{3}}.
\end{equation*}
The total angular velocity of the gLT is found by adding the rotation around the symmetry axis $\mathbf{\hat{3}}$: 
\begin{equation*}
\boldsymbol{\omega}=\boldsymbol{\omega}_{\textrm{ref}}+\dot{\psi}\mathbf{\hat{3}}=-\dot{\varphi}\sin\theta\mathbf{\hat{1}}+\dot{\theta}\mathbf{\hat{2}}+(\dot{\psi}+\dot{\varphi}\cos\theta)\mathbf{\hat{3}}.
\end{equation*}
We shall refer to the third component of this vector as $\omega_3=\dot{\psi}+\dot{\varphi}\cos\theta$. The kinematic equations giving the rotation of the axes $(\mathbf{\hat{1}},\mathbf{\hat{2}},\mathbf{\hat{3}})$ will then be
\begin{equation*}
\begin{array}{lll}&\mathbf{\dot{\hat{1}}}=\boldsymbol{\omega}_{\textrm{ref}}\times\mathbf{\hat{1}}=\dot{\varphi}\cos\theta\mathbf{\hat{2}}-\dot{\theta}\mathbf{\hat{3}},\\
&\mathbf{\dot{\hat{2}}}=\boldsymbol{\omega}_{\textrm{ref}}\times\mathbf{\hat{2}}=-\dot{\varphi}\cos\theta\mathbf{\hat{1}}-\dot{\varphi}\sin\theta\mathbf{\hat{3}},\\
&\mathbf{\dot{\mathbf{\hat{3}}}}=\boldsymbol{\omega}_{\textrm{ref}}\times\mathbf{\hat{3}}=\boldsymbol{\omega}\times\mathbf{\hat{3}}=\dot{\theta}\mathbf{\hat{1}}+\dot{\varphi}\sin\theta\mathbf{\hat{2}}.\end{array}
\end{equation*}
In this notation we can rewrite the reduced equations of motion
\begin{equation*}
\begin{array}{ll}\left(\mathbb{I}\boldsymbol{\omega}\right)^{\mathbf{\dot{}}}=\mathbf{a}\times\left(g_n\hat{z}-\mu g_n\mathbf{v}_A\right),\\ m\mathbf{\ddot{r}}=-\mu g_n\mathbf{v}_A,\end{array}
\end{equation*}
using the Euler angles. We only need to add that the vertical axis is written as $\hat{z}=-\sin\theta\mathbf{\hat{1}}+\cos\theta\mathbf{\hat{3}}$ and that the velocity of the point of support is then $\mathbf{v}_{A}=\nu_x\cos\theta\mathbf{\hat{1}}+\nu_y \mathbf{\hat{2}}+\nu_x\sin\theta\mathbf{\hat{3}}$, where $\nu_x, \nu_y$ are components in the $\hat{x}=\mathbf{\hat{2}}\times\hat{z}$ and $\hat{y}=\mathbf{\hat{2}}$ direction (note here that $\hat{z}\cdot\mathbf{v}_A=0$ as expected). So substituting $\boldsymbol{\omega},\mathbf{v}_A$ and $\hat{z}$ in the equations above with their Euler angle form, differentiating and separating for each component leads us to equations of motion expressed in Euler angles.

We will rewrite equations of motion \eqref{gLT_EoM} in coordinate form in terms of Euler angles.
This formulation has the advantage of being independent of the reference point, be it the $CM$ or the supporting point $A$.
By solving this system for the functions $(\ddot{\theta},\ddot{\varphi},\dot{\omega}_3,\dot{\nu}_x,\dot{\nu}_y)$ we get:
\begin{align}
\label{gLT_ddth}\ddot{\theta}=&\frac{1}{I_1}\left(I_1\dot{\varphi}^2\sin\theta\cos\theta-I_3\omega_3\dot{\varphi}\sin\theta+l\mu g_n\nu_x\cos\theta+lg_n\sin\theta\right)\\
\label{gLT_ddph}\ddot{\varphi}=&\frac{1}{I_1\sin\theta}\left(I_3\omega_3\dot{\theta}-2I_1\dot{\theta}\dot{\varphi}\cos\theta+l\mu g_n\nu_y\right),\\
\label{gLT_dom}\dot{\omega}_3=&0,\\
\label{gLT_dvx}\dot{\nu}_x=&\frac{l\sin\theta}{I_1}\left(I_3\omega_3\dot{\varphi}\cos\theta+I_1(\dot{\theta}^2+\dot{\varphi}^2\sin^2\theta)-lg_n\cos\theta\right)\nonumber\\
&-\frac{\mu g_n\nu_x}{mI_1}\left(I_1+ml^2\cos^2\theta\right)+\nu_y\dot{\varphi},\\
\label{gLT_dvy}\dot{\nu}_y=&-\frac{lI_3\omega_3\dot{\theta}}{I_1}-\frac{I_{1}^{*}}{mI_1}\mu g_n\nu_y-\nu_x\dot{\varphi}.
\end{align}
The equation for $g_n$ \eqref{gn_gLT} becomes
\begin{equation*}
g_n(t)=\frac{mgI_1-ml\left[I_1\cos\theta(\dot{\theta}^2+\dot{\varphi}^2\sin^2\theta)-I_3\omega_3\dot{\varphi}\sin^2\theta\right]}{I_1+ml^2\sin^2\theta+ml^2\mu\nu_x\sin\theta\cos\theta}.
\end{equation*} 
The equation $\dot{\omega}_3=0$ says that $L_\mathbf{\hat{3}}=I_3\omega_3$ is an integral of motion, but $\mathbf{L}\cdot\hat{z}=I_{1}\dot{\varphi}\sin^2\theta+I_3\omega_3\cos\theta$ is not as
\begin{align*}
\frac{d}{dt}\left(\mathbf{L}\cdot\hat{z}\right)=&\frac{d}{dt}\left(I_{1}\dot{\varphi}\sin^2\theta+I_3\omega_3\cos\theta\right)=I_1\ddot{\varphi}\sin^2\theta+2I_1\dot{\varphi}\dot{\theta}\sin\theta\cos\theta-I_3\omega_3\dot{\theta}\sin\theta\nonumber\\
&=l\mu g_n\nu_y\sin\theta.
\end{align*}
The energy for the gliding LT is
\begin{align*}
E=&\frac{1}{2}m\mathbf{\dot{s}}^2+\frac{1}{2}\boldsymbol{\omega}\cdot\mathbf{L}+mg\mathbf{s}\cdot\hat{z}\nonumber\\
=&\frac{1}{2}m\mathbf{v}_{A}^2+ml\mathbf{v}_A\cdot(\boldsymbol{\omega}\times\mathbf{\hat{3}})+\frac{1}{2}ml^2(\boldsymbol{\omega}\times\mathbf{\hat{3}})^2+\frac{1}{2}\boldsymbol{\omega}\cdot\mathbf{L}+mg\mathbf{s}\cdot\hat{z}\nonumber\\
=&\frac{1}{2}m(\nu_{x}^2+\nu_{y}^2)+ml(\nu_{x}\dot{\theta}\cos\theta+\nu_{y}\dot{\varphi}\sin\theta)+\frac{1}{2}(I_{1}^{*}(\dot{\theta}^2+\dot{\varphi}^2\sin^2\theta)+I_3\omega_{3}^2)+mgl\cos\theta.
\end{align*}
Its derivative gives $\dot{E}=\mathbf{F}\cdot\mathbf{v}_A$ when we use equations \eqref{gLT_ddth}--\eqref{gLT_dvy} and the second derivative $\frac{g_n}{m}-g+l(\ddot{\theta}\sin\theta+\dot{\theta}^2\cos\theta)=0$ of the contact criterion.

\section{Asymptotic solutions to gliding Lagrange top}

The dynamical system of equations \eqref{gLT_ddth}--\eqref{gLT_dvy} is nonintegrable and the only solutions that can be directly seen by inspection are $\mathbf{y}_{0,\pi}=(\theta=0,\pi;\dot{\theta}=0,\dot{\varphi}=\text{const},\omega_3=\text{const},\nu_x=0,\nu_y=0)$. These solutions are, as we shall show below, asymptotic solutions to gLT in the sense of the LaSalle theorem.

For an autonomous system in a domain $D\subset\mathbb{R}^n$
\begin{equation}  
\label{dyn_sys}\mathbf{\dot{y}}=\mathbf{Y(y)},
\end{equation}
where $\mathbf{Y}:D\to\mathbb{R}^n$ is a \emph{continuous, Lipschitz map}, a set $M\subset\mathbb{R}^n$ is positively invariant w.r.t. \eqref{dyn_sys} if $\mathbf{y}(0)\in M$ implies that the solution $\mathbf{y}(t)\in M$ for all $t\geq 0$. The LaSalle theorem \cite{LaS,Kha} states

\begin{theorem}[LaSalle]
Suppose $\Omega\subset D$ is compact and positively invariant set for \eqref{dyn_sys}. Let $V$ be a scalar $C^{1}$-function on $D$ and suppose $\dot{V}(\mathbf{y})\leq 0$ in $\Omega$. Let $B$ be the largest invariant set in $M=\{\mathbf{y}\in\Omega:\;\dot{V}(\mathbf{y})=0\}$. Then every solution starting in $\Omega$ approaches $B$ as $t\to\infty$. 
\end{theorem}

For the gLT the energy is a good LaSalle function since $\dot{E}=\mathbf{F}\cdot\mathbf{v}_A=-\mu g_n |\mathbf{v}_{A}|^2\leq 0$ whenever $g_n\geq 0$. 

We consider then equations \eqref{gLT_EoM} defined as a dynamical system for $\mathbf{y}=(\mathbf{\dot{r}},\mathbf{L},\mathbf{\hat{3}})\subset D\subset\mathbb{R}^2\times\mathbb{R}^3\times S^2$. The required asymptotic set is then defined as $M=\{\mathbf{y}=(\mathbf{\dot{r}},\mathbf{L},\mathbf{\hat{3}}):\; \dot{E}=-\mu g_n|\mathbf{v}_A|^2=0\}$. To $M$ belong solutions defined by the condition $\mathbf{v}_A=0$ and, possibly, solutions satisfying $g_n=0$. As we show below, there are no solutions satisfying $g_n(t)=0$.

The definition of the set $M=\{\mathbf{y}=(\mathbf{\dot{r}},\mathbf{L},\mathbf{\hat{3}}):\; \dot{E}=-\mu g_n|\mathbf{v}_A|^2=0\}$ corresponds well to the usual mechanical intepretation of asymptotic solutions understood as frictionless solutions without loss of energy. 

When $g_n\geq 0$ the same energy function defines a positively invariant compact set $\Omega\subset D$ as
\begin{equation*}
\Omega=\left\{(\mathbf{\dot{r}},\mathbf{L},\mathbf{\hat{3}}):E(\mathbf{\dot{r}},\mathbf{L},\mathbf{\hat{3}})=\frac{1}{2}\mathbf{\dot{r}}^2+\frac{1}{2}\left(\frac{1}{I_1}(\mathbf{L}\times\mathbf{\hat{3}})\cdot\hat{z}\right)^2+\frac{1}{2}\boldsymbol{\omega}\cdot\mathbf{L}+mg\mathbf{s}\cdot\hat{z}\leq E(\mathbf{y}(0))+2mgl\right\}.
\end{equation*}
From the constraint $(\mathbf{s}-l\mathbf{\hat{3}})\cdot\hat{z}=0$ we have $-l\leq\mathbf{s}\cdot\hat{z}\leq l$ and the energy is bounded from above and from below. To ensure applicability of the LaSalle arguments we shall need to assume that we consider solutions with nonnegative reaction force $g_n\hat{z}\geq 0$. This is important because, for small $\dot{\varphi}$, $\nu_x$, small angles $\theta$ and sufficiently large $\dot{\theta}$, the numerator in formula \eqref{gn_gLT} for $g_n$ can become negative $gI_1-l(I_1\cos\theta(\dot{\theta}^2+\dot{\varphi}^2\sin^2\theta)-I_3\omega_3\dot{\varphi}\sin^2\theta)<0$. The assumption is also necessary since solutions with initially positive $g_n(0)\geq 0$ may, in principle, acquire negative values of $g_n(t)$ at some later time.

\begin{lemma}
When $g_n(t)>0$ the only asymptotic solutions of the gLT equations \eqref{gLT_EoM} with $\mathbf{v}_A=0$ are the upright and the upside-down spinning solutions
\begin{equation*}
\mathbf{y}_{0,\pi}=(\theta=0,\pi; \dot{\theta}=0,\dot{\varphi}=\text{const},\omega_3=\text{const},\nu_x=0,\nu_y=0).
\end{equation*}
\end{lemma}
\begin{proof}
Asymptotic solutions $\mathbf{v}_A=0$ satisfy 
\begin{equation}
\label{EOM_gLT_red}\begin{array}{l}\left(\mathbb{I}\boldsymbol{\omega}\right)^{\mathbf{\dot{}}}=g_n\mathbf{a}\times\hat{z},\\ \mathbf{\dot{\mathbf{\hat{3}}}}=\boldsymbol{\omega}\times\mathbf{\hat{3}},\\m\mathbf{\ddot{r}}=0.\end{array}.
\end{equation}
When written in Euler angles, equations \eqref{EOM_gLT_red} turns into three equations of motion \eqref{Euler_gLT_1}--\eqref{Euler_gLT_3} and two constraint equations \eqref{Euler_gLT_4}, \eqref{Euler_gLT_5}:
\begin{align}
\label{Euler_gLT_1}& -I_1\ddot{\varphi}\sin\theta-2I_1\dot{\theta}\dot{\varphi}\cos\theta+I_3\dot{\theta}\omega_3=0,\\
\label{Euler_gLT_2}& I_1\ddot{\theta}-I_1\dot{\varphi}^2\sin\theta\cos\theta+I_3\omega_3\dot{\varphi}\sin\theta=lg_n\sin\theta,\\
\label{Euler_gLT_3}& I_3\dot{\omega}_3=0,\\
&\label{Euler_gLT_4} ml(\ddot{\theta}\cos\theta-\sin\theta(\dot{\theta}^2+\dot{\varphi}^2))=0,\\
&\label{Euler_gLT_5} ml(-\ddot{\varphi}\sin\theta-2\dot{\varphi}\dot{\theta}\cos\theta)=0,
\end{align}
where $g_n$ is
\begin{equation*}
g_n=\frac{mgI_1-mlI_1(\cos\theta(\dot{\theta}^2+\dot{\varphi}^2\sin^2\theta)-I_3\omega_3\dot{\varphi}\sin^2\theta)}{I_{1}+ml^2\sin^2\theta}.
\end{equation*}
We substitute the equations of motion into the constraint equations to get the following conditions:
\begin{align}
\label{gLT_constr_1}&\sin\theta(I_1(\dot{\theta}^2+\dot{\varphi}^2\sin^2\theta)+I_3\omega_3\dot{\varphi}\cos\theta-lg_n\cos\theta)=0,\\
\label{gLT_constr_2}&I_3\omega_3\dot{\theta}=0.
\end{align}
These conditions determine the admissible types of solutions to the system \eqref{EOM_gLT_red}. We show that the constraints imply that $\sin\theta=0$.
 The equations above hold if any factor is equal to zero. Suppose $\omega_3=0$ in \eqref{gLT_constr_2}. Then either $\sin\theta=0$ in \eqref{gLT_constr_1} or
\begin{equation*}
I_1(\dot{\theta}^2+\dot{\varphi}^2\sin^2\theta)=lg_n\cos\theta,
\end{equation*}
where
\begin{equation*}
g_n=\frac{mgI_1-mlI_1\cos\theta(\dot{\theta}^2+\dot{\varphi}^2\sin^2\theta)}{I_1+ml^2\sin^2\theta}.
\end{equation*}
We can see that $(\dot{\theta}^2+\dot{\varphi}^2\sin^2\theta)\cos^2\theta$ is constant since its derivative is:
\begin{align*}
&2(\ddot{\theta}\dot{\theta}+\ddot{\varphi}\dot{\varphi}\sin^2\theta+\dot{\varphi}^2\dot{\theta}\sin\theta\cos\theta)\cos^2\theta-2\dot{\theta}\cos\theta\sin\theta(\dot{\theta}^2+\dot{\varphi}^2\sin^2\theta)\\
&=2(\dot{\theta}\dot{\varphi}^2\cos\theta\sin\theta+\frac{lg_n}{I_1}\dot{\theta}\sin\theta-2\dot{\varphi}^2\dot{\theta}\cos\theta\sin\theta+\dot{\varphi}^2\dot{\theta}\sin\theta\cos\theta)\cos^2\theta\\
&-2\dot{\theta}\cos\theta\sin\theta\frac{lg_n}{I_1}\cos\theta=0.
\end{align*}
Call this constant $C$. If we use this constant in the equations above we get
\begin{equation*}
C=\frac{lg_n}{I_1}\cos^3\theta,\qquad
g_n\cos\theta=\dfrac{mgI_1\cos\theta-mlI_1 C}{I_1+ml^2\sin^2\theta}.
\end{equation*}
By eliminating $g_n$ we get a polynomial equation with real coefficients for the unknown $\cos\theta$:
\begin{equation*}
\cos^2\theta(I_1C+2ml^2C-ml^2C\cos^2\theta-mgl\cos\theta)=0.
\end{equation*}
It has at least one real solution. Thus $\cos\theta$ is constant and $\dot{\theta}=0$. If $\omega_3\neq 0$ then by \eqref{gLT_constr_2} again $\dot{\theta}=0$.
 For solutions to \eqref{Euler_gLT_1}-\eqref{Euler_gLT_3} with the constraints \eqref{gLT_constr_1} and \eqref{gLT_constr_2} we have found that $\theta$ is constant. We now show that the only solutions allowed by this system are the upright and inverted spinning gLT, solutions such that $\theta=0$ or $\theta=\pi$. 
 Suppose that $\theta\in(0,\pi)$ (so $\sin\theta\neq 0$). The first constraint equation \eqref{gLT_constr_1} gives 
\begin{equation}
\label{constr_theta_const}I_1\dot{\varphi}^2\sin^2\theta+I_3\omega_3\dot{\varphi}\cos\theta-lg_n\cos\theta=0.
\end{equation}
We have from equation \eqref{Euler_gLT_1} that $\ddot{\varphi}=0$, so that $\dot{\varphi}$ is constant. From equation \eqref{Euler_gLT_2} we get
\begin{equation*}
I_3\omega_3\dot{\varphi}-I_1\dot{\varphi}^2\cos\theta=lg_n,
\end{equation*}
and together with \eqref{constr_theta_const} we obtain
\begin{equation*}
I_1\dot{\varphi}^2=0.
\end{equation*}
So $\dot{\varphi}=0$. But this means in equation \eqref{Euler_gLT_2} that $g_n\sin\theta=0$ which contradicts the assumption $\sin\theta\neq0$. We conclude that for the asymptotic solutions to the gliding LT system we have either $\theta=0$ or $\theta=\pi$. For these solutions we have $g_n=mg$.
\end{proof}
\begin{lemma}
There are no solutions of the system \eqref{gLT_EoM} satisfying $g_n(\mathbf{y})=0$ and $M=\{\mathbf{y}=(\mathbf{\dot{r}},\mathbf{L},\mathbf{\hat{3}}):\;\mathbf{v}_A=0\}=\{\mathbf{y}_0,\mathbf{y}_\pi\}$ is the largest invariant set in $\Omega$.
\end{lemma}
\begin{proof}
Assume that $\mathbf{\tilde{y}}(t)$ is a solution of the gLT equations \eqref{gLT_EoM} such that $g_n(\mathbf{\tilde{y}})=0$ and derive a contradiction. If $g_n=0$ in this system we have that $\mathbf{L}$ is constant and that the numerator in the formula \eqref{gn_gLT} for $g_n$ vanishes:
\begin{equation}
\label{gn_is_zero}\frac{gI_{1}^2}{l}+\mathbf{L}\cdot\hat{z}L_{\mathbf{\hat{3}}}=\mathbf{\hat{3}}_{\hat{z}}L^2.
\end{equation}
If $\mathbf{L}=0$ then the equation reads $\frac{gI_{1}^2}{l}=0$, which is clearly false. If $\mathbf{L}\neq 0$ then $\mathbf{\hat{3}}_{\hat{z}}$ is constant and we can find the other components of $\mathbf{\hat{3}}$ from
\begin{equation*}
\begin{array}{l} 1=\mathbf{\hat{3}}^2,\\ 0=\mathbf{\dot{\mathbf{\hat{3}}}}\cdot\hat{z}=\frac{1}{I_{1}}(L_{\hat{x}}\mathbf{\hat{3}}_{\hat{y}}-L_{\hat{y}}\mathbf{\hat{3}}_{\hat{x}}).\end{array}
\end{equation*}
Either $L_{\hat{x}}$ and $L_{\hat{y}}$ are both zero or both nonzero. In the first case we get $\mathbf{L}=\mathbf{L}\cdot\hat{z}\mathbf{\hat{3}}_{\hat{z}}$, which means for \eqref{gn_is_zero} that $\frac{gI_{1}^2}{l}=0$, which is a contradiction. If $L_{\hat{x}}\neq 0$ and $L_{\hat{y}}\neq 0$ we can solve the system above, which means that $\mathbf{\hat{3}}$ is constant. But then we can write $\mathbf{L}=L_{\mathbf{\hat{3}}}\mathbf{\hat{3}}$ and we find again for \eqref{gn_is_zero} the contradiction $\frac{gI_{1}^2}{l}=0$.
Since we cannot find any solutions $\mathbf{\tilde{y}}(t)$ to the system \eqref{gLT_EoM} such that $g_n(\mathbf{\tilde{y}}(t))=0$ and $M$ is the invariant manifold of solutions to this system such that $\mathbf{v}_{A}=0$, the set $M$ is the largest invariant set in $\Omega$.
\end{proof}
With this lemma we can give a proof of the asymptotic behaviour of trajectories of the gLT system.
\begin{proposition}
Every solution $\mathbf{\tilde{y}}(t)$ of the gLT equations \eqref{gLT_EoM} satisfying the assumption $g_n(t)\geq 0$, $t\geq 0$ goes asymptotically to exactly one of the solutions $\mathbf{y}_{0,\pi}=(\theta=0,\pi;\theta=0,\dot{\varphi}=\text{const},\omega_3=\text{const},\nu_x=0,\nu_y=0)$ .
\end{proposition}
\begin{proof}
We consider $\mathbf{\tilde{y}}(t)$ such that  $g_n(t)\geq 0$, $t\geq 0$ and let $L_{\mathbf{\tilde{y}}}^{+}$ be its positive limit set, i.e. the set of all limit points of $\mathbf{\tilde{y}}(t)$ for sequences $\{t_m\}$ such that $t_m\to\infty$ as $m\to\infty$.
As we have already mentioned, $E(\mathbf{\tilde{y}}(t))$ is a decreasing function of $t$, so $\mathbf{\tilde{y}}(t)$ is contained in the compact set $\Omega=\{\mathbf{\tilde{y}}(t)\in D: E(\mathbf{\tilde{y}}(t))\leq E(\mathbf{\tilde{y}}(0))\}$. We have $L_{\mathbf{\tilde{y}}}^{+}\subset \Omega$ since $\Omega$ is closed.
 The function $E(\mathbf{\tilde{y}}(t))$ has a limit $a$ as $t\to\infty$ since $E(\mathbf{\tilde{y}}(t))$ is continuous on the compact set $\Omega$. One can show that since $\mathbf{\tilde{y}}(t)$ is bounded on the compact set $\Omega$ the positive limit set $L_{\mathbf{\tilde{y}}}^{+}$ is nonempty, compact and invariant (see \cite{Kha}, appendix), so if $p\in L_{\mathbf{\tilde{y}}}^{+}$ then there is a sequence $\{t_m\}$ with $t_m\to\infty$ and $\mathbf{\tilde{y}}(t_m)\to p$ as $m\to\infty$. By the continuity of $E$, we have $E(p)=\lim_{m\to\infty}E(\mathbf{\tilde{y}}(t_m))=a$.
The energy $E$ is thus constant on $L_{\mathbf{\tilde{y}}}^{+}$ and since $L_{\mathbf{\tilde{y}}}^{+}$ is an invariant set, $\dot{E}=0$ on $L_{\mathbf{\tilde{y}}}^{+}$.
 By the previous lemma, $B$ is the largest invariant set in $\{\mathbf{\tilde{y}}\in D: \dot{E}(\mathbf{\tilde{y}})=0\}$, so it follows that $L_{\mathbf{\tilde{y}}}^{+}\subset B$. By inclusion, $\mathbf{\tilde{y}}(t)$ approaches $B$ as $t\to\infty$.
Since $B$ only contains two isolated solutions and $L_{\mathbf{\tilde{y}}}^{+}$ is connected, we see that the positive limit set must coincide with one of the solutions for \eqref{EOM_gLT_red}.  
\end{proof}
The statement of this proposition reflects the mechanical understanding that when $t\to\infty$ the energy decreases the gliding velocity $\mathbf{v}_A\to 0$ and the asymptotic solution is one of the stationary solutions of the classical LT.

It remains to investigate the relative stability of these asymptotic solutions.
The energy of the gliding HST can be rewritten as a sum of two terms:
\begin{align}
E=&\left(\frac{1}{2}m(\nu_{x}^2+\nu_{y}^2)+ml(\nu_x\dot{\theta}\cos\theta+\nu_y\dot{\varphi}\sin\theta)+\frac{1}{2}I_{1}^{*}\dot{\theta}^2\right)+\nonumber\\
&+\left(\frac{(\mathbf{L}\cdot\mathbf{\hat{3}})^2}{2I_3}+\frac{(\mathbf{L}_{A}\cdot\hat{z}(t)-\mathbf{L}\cdot\mathbf{\hat{3}}\cos\theta)^2}{2I_{1}^*(1-\cos^2\theta)}+mgl\cos\theta\right)\nonumber\\
=&E_{1}(\theta,\dot{\theta},\dot{\varphi},\nu_x,\nu_y)+E_{2}(\theta,\mathbf{L}_{A}\cdot\hat{z}(t),\mathbf{L}\cdot\mathbf{\hat{3}}).
\end{align}
The first function vanishes for the asymptotic solutions (since then $\mathbf{v}_A\to0$ and $\dot{\theta}=0$) the second function goes to
\begin{equation}
E_{2}(\theta,\mathbf{L}_{A}\cdot\hat{z},\mathbf{L}\cdot\mathbf{\hat{3}})=E_{2}(\theta)=\frac{(\mathbf{L}\cdot\mathbf{\hat{3}})^2}{2I_3}+\frac{(\mathbf{L}_{A}\cdot\hat{z}-\mathbf{L}\cdot\mathbf{\hat{3}}\cos\theta)^2}{2I_{1}^*(1-\cos^2\theta)}+mgl\cos\theta,
\end{equation}
because $\mathbf{L}_{A}\cdot\hat{z}$ is an integral of motion for the HST. The problem of checking the stability of the asymptotic solutions is reduced to examining the character of extremal values of $E_{2}(\theta)$ for the points $\theta=0$ and $\theta=\pi$. These solutions are asymptotically stable if $E_{2}^{\prime\prime}(\theta)>0$ for $\theta=0$ and $\theta=\pi$. But we have
\begin{equation}
E_{2}^{\prime\prime}(\theta)=-\cos\theta E_{2}^{\prime}(\cos\theta)+(1-\cos^{2}\theta)E_{2}^{\prime\prime}(\cos\theta),
\end{equation}
so we see that we only have to investigate $E_{2}^{\prime}(\cos\theta)$ for $\cos\theta=\pm 1$.
 We look first at the asymptotic solution $\cos\theta=1$. For this solution $\mathbf{L}_{A}\cdot\hat{z}=\mathbf{L}\cdot\mathbf{\hat{3}}=\mathbf{L}_{A}\cdot\mathbf{\hat{3}}$, so the derivative of $E_{2}(\cos\theta)$ is equal to
\begin{equation}
E_{2}^{\prime}(\cos\theta)=\frac{-(\mathbf{L}_{A}\cdot\mathbf{\hat{3}})^2}{I_{1}^{*}(1+\cos\theta)^2}+mgl.
\end{equation} 
We then get 
\begin{equation}
E_{2}^{\prime\prime}(\theta=0)>0 \Leftrightarrow E_{2}^{\prime}(\cos\theta=1)<0 \Leftrightarrow (\mathbf{L}_{A}\cdot\mathbf{\hat{3}})^2> 4mglI_{1}^{*}.
\end{equation}
Thus the upright spinning solution is stable if the angular momentum about the $\mathbf{\hat{3}}$-axis satisfies $|\mathbf{L}_{A}\cdot\mathbf{\hat{3}}|>2\sqrt{mglI_{1}^*}$.
 The second asymptotic solution is $\cos\theta=-1$. For this solution $\mathbf{L}_{A}\cdot\hat{z}=-\mathbf{L}_{A}\cdot\mathbf{\hat{3}}$, so the derivative of $E_{2}(\cos\theta)$ is
\begin{equation}
E_{2}^{\prime}(\cos\theta)=\frac{(\mathbf{L}_{A}\cdot\mathbf{\hat{3}})^2}{I_{1}^{*}(1-\cos\theta)^2}+mgl.
\end{equation}
From this we get
\begin{equation}
E_{2}^{\prime\prime}(\theta=\pi)>0\Leftrightarrow E_{2}^{\prime}(\cos\theta=-1)>0 \Leftrightarrow \frac{(\mathbf{L}_{A}\cdot\mathbf{\hat{3}})^2}{4I_{1}^*}+mgl>0.
\end{equation}
Clearly this inequality is always satisfied, so the upside-down spinning position is always stable. The conditions of stability for the asymptotic straight and inverted spinning solutions are the same as for the classic Lagrange top \cite{Arn}.

\bibliography{mybib}
\end{document}